\documentclass[11pt]{scrartcl}

\usepackage[normalem]{ulem}

\usepackage{amsmath,amssymb,latexsym,soul,amsthm}
\usepackage{cite,amsfonts}
\usepackage{color,enumitem,graphicx}
\usepackage[colorlinks=true,urlcolor=blue,
=red,linkcolor=blue,linktocpage,pdfpagelabels,
bookmarksnumbered,bookmarksopen]{hyperref}

\usepackage{authblk}

\usepackage[a4paper,left=1.0in,right=1.0in,top=3cm,bottom=3cm]{geometry}

\usepackage{authblk}

\usepackage[a4paper,left=1.0in,right=1.0in,top=3cm,bottom=3cm]{geometry}

\usepackage[T1]{fontenc}
\usepackage{lmodern}
\usepackage{microtype}

\usepackage{fixmath}

\numberwithin{equation}{section}

\newtheorem{theorem}{Theorem}[section]
\newtheorem{lemma}[theorem]{Lemma}
\newtheorem{proposition}[theorem]{Proposition}
\newtheorem{corollary}[theorem]{Corollary}
\theoremstyle{definition}
\newtheorem{definition}[theorem]{Definition}
\theoremstyle{remark}
\newtheorem{remark}[theorem]{Remark}

\usepackage{todonotes}

\numberwithin{equation}{section}

\DeclareSymbolFont{rsfs}{U}{rsfs}{m}{n}
\DeclareSymbolFontAlphabet{\mathscr}{rsfs}

\begin{document}

\title{Ground state solutions for Bessel fractional equations with irregular nonlinearities}

\author{Simone Secchi\thanks{Email address: \texttt{Simone.Secchi@unimib.it}}}
\affil{\small Dipartimento di Matematica e Applicazioni \\ Universit\`a degli Studi di Milano-Bicocca \\ via Roberto Cozzi 55, I-20125, Milano, Italy}

\maketitle

\begin{abstract}
We consider the semilinear fractional equation 
\( (I-\Delta)^s u = a(x) |u|^{p-2}u\) in \(\mathbb{R}^N\), where \(N \geq 3\), \(0<s<1\), \(2<p<2N/(N-2s)\) and \(a\) is a bounded weight function. Without assuming that \(a\) has an asymptotic profile at infinity, we prove the existence of a ground state solution.

\medskip


\end{abstract}

\null
\begin{flushright}
\textit{Dedicated to Francesca}
\end{flushright}

\tableofcontents

\section{Introduction}

To pursue further the study that we began in \cite{Secchi17,Secchi17-1}, we consider in this paper the equation
\begin{equation} \label{eq:1.1} (I-\Delta)^s u = a(x) |u|^{p-2}u \quad
  \hbox{in \(\mathbb{R}^N\)},
\end{equation}
where \(a \in L^\infty(\mathbb{R}^N)\), \(N>2\), \(0<s<1\)  and $2<p<2_s^\star = 2N/(N-2s)$.

When \(s=1\), \eqref{eq:1.1} formally reduces to the semilinear elliptic equation
\begin{align*}
	-\Delta u + u = a(x)|u|^{p-2}u,
\end{align*}
which has been widely studied over the years. This equation can be seen as a particular case of the stationary Nonlinear Schr\"{o}dinger Equation
\begin{align}
	-\Delta u + V(x)u = a(x)|u|^{p-2}u \quad \hbox{in $\mathbb{R}^N$}. \label{eq:1.2}
\end{align}
When both $V$ and $a$ are constants, we refer to the seminal papers \cite{BerestyckiLionsI,BerestyckiLionsII} and to the references therein. Since the \emph{non-compact} group of translations acts on \(\mathbb{R}^N\), when $V$ and $a$ are general functions the analysis becomes subtler, and solutions exist according to some properties of these potentials. For instance, when both $V$ and $a$ are radially symmetric, \eqref{eq:1.2} is invariant under rotations, and it becomes legitimate to look for radially symmetric solutions: see \cite{DingNi}.

Without any \emph{a priori} symmetry assumption, the lack of compactness in \eqref{eq:1.2} must be overcome with a careful analysis, and the behavior of $V$ and $a$ at infinity plays a crucial r\^{o}le. The first attempt to solve \eqref{eq:1.2} in the case $\lim_{|x| \to +\infty} V(x)=+\infty$ and $a$ is a constant appeared in \cite{Rabinowitz}. With similar techniques, it is possible to solve \eqref{eq:1.2} under the  assumption $\limsup_{|x| \to +\infty} a(x) \leq 0$. So many papers dealing with \eqref{eq:1.2} (or with even more general equations) appeared in the literature afterwards that we refrain from any attempt to give a complete overview.

\bigskip

If $0<s<1$, our equation becomes \emph{non-local}, since the fractional power $(I-\Delta)^s$ of the positive operator $I-\Delta$ in  $L^2(\mathbb{R}^N)$ is no longer a differential operator. It is strictly related to the more popular \emph{fractional laplacian} $(-\Delta)^s$, but it behaves worse under scaling.  We offer a very quick review of this operator.

\bigskip

For $s>0$ we introduce the \emph{Bessel function space}
\[ 
L^{s,2}(\mathbb{R}^N) = \left\{ f \in L^2(\mathbb{R}^N)  \mid f=G_s \star g \
\text{for some $g \in L^2(\mathbb{R}^N)$} \right\},
\] 
where the Bessel convolution kernel is defined by
\begin{equation*} 
G_s (x) = \frac{1}{(4 \pi )^{s
/2}\Gamma(s/2)} \int_0^\infty \exp \left( -\frac{\pi}{t} |x|^2
\right) \exp \left( -\frac{t}{4\pi} \right) t^{\frac{s - N}{2}-1}
\, dt.
\end{equation*} 
The Bessel space is endowed with the norm~$\|f\| = \|g\|_2$ if
$f=G_s \star g$. The operator $(I-\Delta)^{-s} u = G_{2s}
\star u$ is usually called Bessel operator of order $s$.

In Fourier variables the same operator reads
\begin{equation*} 
G_s = \mathcal{F}^{-1} \circ
\left( \left(1+|\xi|^2 \right)^{-s /2} \circ \mathcal{F} \right),
\end{equation*} 
so that
\[ 
\|f\| = \left\| (I-\Delta)^{s /2} f \right\|_2.
\] 
For more detailed information, see \cite{Adams, Stein} and the
references therein.

In the paper \cite{Fall} the pointwise formula
\begin{equation*} 
(I-\Delta)^s u(x) = c_{N,s}
\operatorname{P.V.} \int_{\mathbb{R}^N}
\frac{u(x)-u(y)}{|x-y|^{\frac{N+2s}{2}}}
K_{\frac{N+2s}{2}}(|x-y|) \, dy + u(x)
\end{equation*} 
was derived for functions $u \in C_c^2(\mathbb{R}^N)$.
Here $c_{N,s}$ is a positive constant depending only on $N$ and
$s$, P.V. denotes the principal value of the singular integral,
and $K_\nu$ is the modified Bessel function of the second kind with
order $\nu$ (see \cite[Remark 7.3]{Fall} for more details).  However
a closed formula for $K_\nu$ is not known.

\smallskip

We summarize the main properties of Bessel spaces. For the proofs
we refer to \cite[Theorem 3.1]{Felmer}, \cite[Chapter V, Section 3]{Stein}.

\begin{theorem} \label{th:sobolev}
	\begin{enumerate}
		\item $L^{s,2}(\mathbb{R}^N) =
W^{s,2}(\mathbb{R}^N) = H^s (\mathbb{R}^N)$, where the sign of equality must be understood in the sense of an isomorphism.
		\item If $s \geq 0$ and $2 \leq q \leq
2_s^*=2N/(N-2s)$, then $L^{s,2}(\mathbb{R}^N)$ is
continuously embedded into $L^q(\mathbb{R}^N)$; if $2 \leq q <
2_s^*$ then the embedding is locally compact.
		\item Assume that $0 \leq s \leq 2$ and $s >
N/2$. If $s -N/2 >1$ and $0< \mu \leq s - N/2-1$, then
$L^{s,2}(\mathbb{R}^N)$ is continuously embedded into
$C^{1,\mu}(\mathbb{R}^N)$. If $s -N/2 <1$ and $0 < \mu \leq
s -N/2$, then $L^{s,2}(\mathbb{R}^N)$ is continuously
embedded into $C^{0,\mu}(\mathbb{R}^N)$.
	\end{enumerate}
\end{theorem}
\begin{remark} 
	According to Theorem \ref{th:sobolev}, the Bessel space $L^{s,2}(\mathbb{R}^N)$
is topologically undistinguishable from the So\-bo\-lev fractional space
$H^s(\mathbb{R}^N)$. Since our equation
involves the Bessel norm, we will not exploit this characterization.
\end{remark}
Going back to \eqref{eq:1.1}, it must be said that in the case $s \in (0,1)$ less is known than in the \emph{local} case $s=1$. Equation~\eqref{eq:1.1} arises from the more general Schr\"{o}dinger-Klein-Gordon equation
\begin{align*}
	\mathrm{i}\frac{\partial \psi}{\partial t} = (I-\Delta)^s \psi -\psi -f(x,\psi)
\end{align*}
describing the the behaviour of bosons, spin-0 particles in relativistic fields. We refer to \cite{FelmerVergara, Secchi17,Secchi17-1,Secchi17-2} for very recent results about the existence of variational solutions.   When $s=1/2$, the operator $(I-\Delta)^{1/2}=\sqrt{\strut I-\Delta}$ is also called \emph{pseudorelativistic} or \emph{semirelativistic}, and it is very important in the study of several physical phenomena.
The interested reader can refer to \cite{CingolaniSecchi15,CingolaniSecchi18} and to the references therein for more information.
\begin{remark}
	The identity operator $I$ is often replaced by a multiple $m^2 I$, for some real number~$m \neq 0$. The operator reads then $(-\Delta+m^2)^s$, but for our purposes this generality does not give any advantage.
\end{remark}
A common feature in the current literature is that the existence of solutions to \eqref{eq:1.1} is related to the behavior of the potential function $a$ at infinity. This is a very useful tool for applying concentration-compactness methods or for working in weighted Lebesgue spaces. In the present paper, following \cite{AckerChagoya}, we investigate \eqref{eq:1.1} under much weaker assumptions on $a$, see Section 2. The first existence results for semilinear elliptic equations with \emph{irregular} potentials appeared, as far as we know, in \cite{CeramiMolle}.

\section{The variational setting}

We introduce some tools that will be used systematically in the rest
of the paper.

\begin{definition}
\begin{itemize}
\item For any $y \in \mathbb{R}^N$, we define the translation operator
  $\tau_y$ acting on a (suitably regular) function $f$ as
  $\tau_y f\colon x \mapsto f(x-y)$.
\item In a normed space $X$, we denote by $B(x,r)$ the ball centered
  at $x \in X$ with radius~$r >0$, and by $\overline{B}(x,r)$ its
  closure. The boundary of~$B(0,1)$ will be denoted by $S(X)$.
\item For any $a \in L^\infty(\mathbb{R}^N)$, we define
\begin{equation*}
\mathscr{P} = \overline{B}(0,|a|_\infty) \subset L^\infty(\mathbb{R}^N).
\end{equation*}
Looking at $L^\infty(\mathbb{R}^N)$ as the dual space of
$L^1(\mathbb{R}^N)$, the set $\mathscr{P}$ will be endowed with the
weak* topology. It is well-known that $\mathscr{P}$ becomes a compact
metrizable space, see \cite[Theorem 3.15 and Theorem 3.16]{Rudin}.
\item For any $a \in L^\infty(\mathbb{R}^N)$, we define the subset
  $\mathscr{A} = \left\{ \tau_y a \mid y \in \mathbb{R}^N \right\}$ of
  $\mathscr{P}$, endowed with the relative topology. Finally, we
  introduce
  $\mathscr{B} = \overline{\mathscr{A}} \setminus \mathscr{A}$.
\item For any $a \in L^\infty(\mathbb{R}^N)$, we define
\begin{equation} \label{eq:a-bar}
\bar{a} = \sup \left\{ \operatorname{ess\, sup} u \mid u \in \mathscr{B} \right\}.
\end{equation}
If $\mathscr{B}=\emptyset$, we agree that $\bar{a} = -\infty$.
\end{itemize}
\end{definition}
The following is the main assumption of the present paper.
\begin{description}
\item[(A)] The function $a \in L^\infty(\mathbb{R}^N)$ is such that $a^{+} = \max\{a,0\}$ is not identically zero, and either (i) $\bar{a} \leq 0$ or (ii) $\bar{a} \leq a$. 
\end{description}
Weak solutions to \eqref{eq:1.1} are critical points of the functional $I_a \colon L^{s,2}(\mathbb{R}^N) \to \mathbb{R}^N$ defined by
\begin{equation*}
I_a(u) = \frac{1}{2} \| u \|_{L^{s,2}}^2 - \frac{1}{p} \int_{\mathbb{R}^N} a |u|^p.
\end{equation*}
\begin{definition}
A solution $u \in L^{s,2}(\mathbb{R}^N)$ is called a ground-state solution to \eqref{eq:1.1} if $I_a$ attains at $u$ the infimum over the set of all solutions to \eqref{eq:1.1}, namely
\begin{equation*}
I_a(u) = \min \left\{ I_a(v) \mid \hbox{$v \in L^{s,2}(\mathbb{R}^N)$ solves \eqref{eq:1.1}} \right\}.
\end{equation*}
\end{definition}
We now state the main result of our paper.
\begin{theorem} \label{th:main}
Equation \eqref{eq:1.1} has (at least) a positive ground state provided that $2<p<2_s^\star$ and $a \in L^\infty(\mathbb{R}^N)$ satisfies \textbf{(A)}.
\end{theorem}

\section{The construction of a Nehari manifold}

We introduce the Nehari set of $I_a$ as
\begin{equation*}
  \mathscr{N}_a = \left\{ u \in L^{s,2}(\mathbb{R}^N) \mid u \neq 0, \ DI_a(u)[u]=0 \right\}.
\end{equation*}
\begin{definition}
  $c_a = \inf_{u \in \mathscr{N}_a} I_a(u)$. We agree that
  $c_a = +\infty$ if $\mathscr{N}_a = \emptyset$.
\end{definition}
To proceed further, we need a ``dual'' characterization of the essential supremum.
\begin{lemma} \label{lem:3.1}
Let \(a \in L^\infty(\mathbb{R}^N)\). There results
\begin{equation} \label{eq:3.1}
\operatorname{ess\, sup} a = \sup \left\{ \int_{\mathbb{R}^N} a \varphi \mid \varphi \in L^1(\mathbb{R}^N), \ \varphi \geq 0, \ \int_{\mathbb{R}^N} \varphi =1 \right\}.
\end{equation}
\end{lemma}
\begin{proof}
Whenever \(\varphi \in L^1(\mathbb{R}^N)\), \(\varphi \geq 0\), \( \int_{\mathbb{R}^N} \varphi =1\), we compute
\begin{equation*}
\int_{\mathbb{R}^N} a \varphi \leq \operatorname{ess\, sup} a \int_{\mathbb{R}^N} \varphi = \operatorname{ess\, sup} a.
\end{equation*}
Hence
\begin{equation} \label{eq:3.2}
\operatorname{ess\, sup} a \geq \sup \left\{ \int_{\mathbb{R}^N} a \varphi \mid \varphi \in L^1(\mathbb{R}^N), \ \varphi \geq 0, \ \int_{\mathbb{R}^N} \varphi =1 \right\}.
\end{equation}
On the other hand, if we set
\begin{equation*}
\sup \left\{ \int_{\mathbb{R}^N} a \varphi \mid \varphi \in L^1(\mathbb{R}^N), \ \varphi \geq 0, \ \int_{\mathbb{R}^N} \varphi =1 \right\} = b
\end{equation*}
and we assume that \(\operatorname{ess\, sup} a >  b\),
then for some \(\delta>0\) we can say  that the set \(\Omega = \left\{ x \in \mathbb{R}^N \mid a(x) \geq b + \delta \right\}\) has positive measure. Let us define \(\varphi = \chi_\Omega / \mathcal{L}^N(\Omega)\), so that 
\begin{equation*}
\int_{\mathbb{R}^N} a \varphi = \frac{1}{\mathcal{L}^N(\Omega)} \int_\Omega a \geq b+\delta,
\end{equation*}
contrary to \eqref{eq:3.2}. This completes the proof.
\end{proof}
Recall from assumption \textbf{(A)} that $a^{+} \neq 0$ as an element
of $L^\infty(\mathbb{R}^N)$. Therefore Lemma \ref{lem:3.1} yields a function
$\varphi \in S(L^1(\mathbb{R}^N))$ such that $\varphi \geq 0$ and
$\int_{\mathbb{R}^N} a \varphi >0$. By a standard mollification
argument, we can assume without loss of generality that
$\varphi \in C_c^\infty(\mathbb{R}^N)$.

\bigskip

Since $L^{s,2}(\mathbb{R}^N)$ is continuously embedded into
$L^p(\mathbb{R}^N)$ for every $2<p<2_s^\star$, we can set
\begin{equation*}
  S_p = \sup \left\{ \frac{|u|_p}{\|u\|_{L^{s,2}}}  \mid u \in L^{s,2}(\mathbb{R}^N), \ u \neq 0  \right\} \in (0,+\infty).
\end{equation*}
We write
\begin{equation*}
  \mathscr{B}_a^{+}= \left\{ u \in L^{s,2}(\mathbb{R}^N) \mid \int_{\mathbb{R}^N} a |u|^p >0  \right\}
\end{equation*}
and
\begin{equation*}
  \mathscr{S}_a^{+} = \mathscr{B}_a^{+} \cap S(L^{s,2}(\mathbb{R}^N)).
\end{equation*}
\begin{lemma}
  The set $\mathscr{B}_a^{+}$ is non-empty and open in
  $L^{s,2}(\mathbb{R}^N)$.
\end{lemma}
\begin{proof}
  We already know that $\varphi \in \mathscr{B}_a^{+}$. Furthermore,
  the map $u \mapsto \int_{\mathbb{R}^N} a|u|^p$ is continuous from
  $L^{s,2}(\mathbb{R}^N)$ to $\mathbb{R}$, since
  $a \in L^\infty(\mathbb{R}^N)$ and $2<p<2^\star$. This immediately
  implies that $\mathscr{B}_a^{+}$ is an open subset of
  $L^{s,2}(\mathbb{R}^N)$.
\end{proof}
\begin{lemma}
  There exists a homeomorphism $\mathscr{S}_a^{+} \to \mathscr{N}_a$
  whose inverse map is $u \mapsto u/\|u\|_{L^{s,2}}$.
\end{lemma}
\begin{proof}
  For any $u \in L^{s,2}(\mathbb{R}^N) \setminus \{0\}$ we consider the
  \emph{fibering map}
	\begin{equation*}
	h(t) = I_a(tu), \qquad (t \geq 0).
	\end{equation*}
	It follows easily that $h$ has a positive critical point if,
        and only if, $u \in \mathscr{B}_a^{+}$. It is a Calculus
        exercise to check that, in this case, the critical point of
        $h$ is the unique non-degenerate global maximum $\bar{t}(u)>0$
        of $h$. By direct computation, $ t u \in \mathscr{N}_a$ if,
        and only if, $t=\bar{t}(u)$. Explicitly,
	\begin{equation*}
	\bar{t}(u) = \frac{\|u\|_{L^{s,2}}^2}{\int_{\mathbb{R}^N} a|u|^p}.
	\end{equation*}
	This shows that the map $u \mapsto \bar{t}(u)$ is continuous from $\mathscr{B}_a^{+}$ to $(0,+\infty)$. The rest of the proof follows easily.
\end{proof}
\begin{lemma}
	The set $\mathscr{N}_a$ is closed in $L^{s,2}(\mathbb{R}^N)$.
\end{lemma}
\begin{proof}
	If $u \in \mathscr{N}_a$, then
	\begin{equation*}
	\|u\|_{L^{s,2}}^2 = \int_{\mathbb{R}^N} a |u|^p \leq \int_{\mathbb{R}^N} a^{+} |u|^p \leq S_p |a^{+}|_\infty \|u\|_{L^{s,2}}^p. 
	\end{equation*}
	It follows that
	\begin{equation} \label{eq:3.3}
	\inf_{u \in \mathscr{N}_a} \|u\|_{L^{s,2}} \geq \frac{1}{S_p |a^{+}|_\infty^{1/(p-2)}}.
	\end{equation}
	As a consequence, $0$ is not a cluster point of $\mathscr{N}_a$, which turns out to be closed.
\end{proof}
It is now standard to invoke the Implicit Function Theorem to prove that $\mathscr{N}_a$ is a $C^2$-submanifold of $L^{s,2}(\mathbb{R}^N)$ and that
\eqref{eq:3.3} implies 
\begin{equation*}
\inf_{u \in \mathscr{N}_a} I_a(u) \geq \left( \frac{1}{2} - \frac{1}{p} \right) \frac{1}{S_p^2 |a^{+}|_\infty^{2/(p-2)}}.
\end{equation*}
More importantly, $\mathscr{N}_a$ is a \emph{natural constraint} for $I_a$, i.e. every critical point of the restriction $\bar{I}_a$ of $I_a$ to $\mathscr{N}_a$ is a nontrivial critical point of $I_a$.
The following result was proved in \cite[Proposition 3.2]{FelmerVergara}, and allows us to consider only positive ground states.
\begin{proposition}
  Any weak solution to \eqref{eq:1.1} is strictly positive.
\end{proposition}
\begin{proposition} \label{prop:3.7}
Let $\bar{I}_a$ be the restriction of the functional $I_a$ to the manifold $\mathscr{N}_a$. Every Palais-Smale sequence at level $c$ for $\bar{I}_a$ is also a
Palais-Smale sequence at level $c$ for $I_a$.
\end{proposition}
\begin{proof}
Assume that $\{u_n\}_n \subset \mathscr{N}_a$ is a Palais-Smale sequence
at level $c$ for $\bar{I}_a$, namely
\begin{align*}
  \lim_{n \to +\infty} \bar{I}_a(u_n) =c
\end{align*}
and
\begin{align*}
  \lim_{n \to +\infty} D\bar{I}_a(u_n) =0
\end{align*}
in the norm topology. It suffices to show
that the sequence $\{ \nabla I_a(u_n)\}_n$ converges to zero in
$L^{s,2}(\mathbb{R}^N)$. Let us abbreviate $\psi(u) = D I_a(u)[u]$, so that
$\mathscr{N}_a = \psi^{-1}(\{0\}) \setminus \{0\}$. From the fact that
$u_n \in \mathscr{N}_a$, we deduce that $I_a(u_n) =
(1/2-1/p)\|u_n\|_{L^{s,2}}^2$, and hence the sequence $\{u_n\}_n$ is
bounded. This implies that
\begin{equation} \label{eq:3.4}
\sup_n \frac{\left\| \nabla \psi (u_n) \right\|_{L^{s,2}}}{\|u_n\|_{L^{s,2}}}
< +\infty.
\end{equation}
Explicitly, we have that, for every $n \in \mathbb{N}$,
\begin{equation} \label{eq:3.5}
\langle \nabla \psi(u_n)\mid u_n \rangle = (2-p) \|u_n\|_{L^{s,2}}^2 <0
\end{equation}
and
\begin{equation} \label{eq:3.6}
  \nabla \bar{I}_a(u_n) = \nabla I_a(u_n) - \frac{\langle \nabla
    I_a(u_n)\mid \nabla \psi (u_n) \rangle}{\|\nabla
    \psi(u_n)\|_{L^{s,2}}^2} \nabla \psi(u_n).
\end{equation}
Observe that $\nabla I_a(u_n) \perp u_n$ because $u_n \in
\mathscr{N}_a$. If we consider the quantity
\begin{equation*}
  \left\| \nabla \psi (u_n) \right\|_{L^{s,2}}^2 - \left( \frac{\langle
      \nabla I_a(u_n)\mid \nabla \psi(u_n) \rangle}{\|\nabla I_a(u_n)\|_{L^{s,2}}^2}
    \right)^2,
  \end{equation*}
  we immediately see that it equals the square of the norm of the
  projection of the vector $\nabla \psi(u_n)$ onto the subspace of
  $L^{s,2}(\mathbb{R}^N)$ orthogonal to the unit vector
  $\nabla I_a(u_n)/\|\nabla I_a(u_n)\|$. Since this subspace contains
  in particular the vector $u_n/\|u_n\|_{L^{s,2}}$, it follows from the Pythagorean
  Theorem that
  \begin{equation}
 \left\| \nabla \psi (u_n) \right\|_{L^{s,2}}^2 - \left( \frac{\langle
      \nabla I_a(u_n)\mid \nabla \psi(u_n) \rangle}{\|\nabla I_a(u_n)\|_{L^{s,2}}^2}
    \right)^2 \geq \left( \frac{\langle \nabla \psi(u_n)\mid u_n
        \rangle}{\|u_n\|_{L^{s,2}}} \right)^2.
  \end{equation}
  This yields, recalling \eqref{eq:3.6}, \eqref{eq:3.5} and \eqref{eq:3.4},
  \begin{align*}
\left\| \nabla \bar{I}_a(u_n) \right\|_{L^{s,2}} \left\| \nabla I_a(u_n)
\right\|_{L^{s,2}} &\geq \langle \nabla \bar{I}_a(u_n)\mid \nabla I_a(u_n)
\rangle \\
&= \frac{\|\nabla I_a(u_n)\|_{L^{s,2}}^2}{\|\nabla \psi(u_n)\|_{L^{s,2}}^2}
\left( \left\| \nabla \psi(u_n) \right\|_{L^{s,2}}^2 - \left(
    \frac{\langle \nabla I_a(u_n)\mid \nabla
      \psi(u_n)\rangle}{\|\nabla I_a(u_n)\|_{L^{s,2}}^2}
    \right)^2
  \right)^2 \\
  &\geq \frac{\|\nabla I_a(u_n)\|_{L^{s,2}}^2}{\|\nabla
    \psi(u_n)\|_{L^{s,2}}^2} \left( \frac{\langle \nabla \psi(u_n)\mid u_n
        \rangle}{\|u_n\|_{L^{s,2}}} \right)^2 \\
        &= \frac{\|\nabla I_a(u_n)\|_{L^{s,2}}^2}{\|\nabla
      \psi(u_n)\|_{L^{s,2}}^2} (2-p)^2 \|u_n\|_{L^{s,2}}^2 \\
    &\geq C \|\nabla I_a(u_n)\|_{L^{s,2}}^2.
  \end{align*}
  This argument proves that
  $\lim_{n \to +\infty} \|\nabla I_a(u_n)\|_{L^{s,2}}=0$, and we conclude.
\end{proof}

\section{Splitting and vanishing sequences}

The analysis of Palais-Smale sequences can be harder than in the more
familiar case of a potential function $a$ that has a precise
asymptotic behavior at infinity. For this reason, we recall a language
taken from \cite{AckerChagoya}.
\begin{definition}
  A map $F \colon X \to Y$ between two Banach spaces splits in the BL
  sense\footnote{BL stands for Brezis and Lieb.} if for any sequence
  $\{u_n\}_n \subset X$ such that $u_n \rightharpoonup u$ in $X$ there
  results
  \begin{equation*}
    F(u_n-u) = F(u_n)-F(u)+o(1)
  \end{equation*}
  in the norm topology of~$Y$.
\end{definition}
\begin{lemma} \label{lem:4.2}
  Suppose that $\{u_n\}_n \subset L^{s,2}(\mathbb{R}^N)$ and
  $\{y_n\}_n \subset \mathbb{R}^N$ are such that
  $\tau_{-y_n}u_n \rightharpoonup u_0$ in $L^{s,2}(\mathbb{R}^N)$. Then
  \begin{equation*}
    I_{\tau_{-y_n}a}(\tau_{-y_n}u_n) - I_{\tau_{-y_n}a} (\tau_{-y_n}u_n-u_0)-I_{\tau_{-y_n}a}(u_0)=o(1)
  \end{equation*}
  and
  \begin{equation*}
    DI_{\tau_{-y_n}a}(\tau_{-y_n}u_n) - DI_{\tau_{-y_n}a} (\tau_{-y_n}u_n-u_0) - DI_{\tau_{-y_n}a}(u_0) =o(1).
    \end{equation*}
  \end{lemma}
  \begin{proof}
Since the maps $F(u)=p^{-1}|u|^p$ and $F'(u)=|u|^{p-2}u$ both split from
$L^{s,2}(\mathbb{R}^N)$ into $L^1(\mathbb{R}^N)$, see \cite[Lemma
4.4]{Secchi17}, we can write
\begin{multline*}
  \int_{\mathbb{R}^N} \left| \left( \tau_{-y_n} a \right) \left( F( \tau_{-y_n}u_n) - F(\tau_{-y_n}u_n-u_0)-F(u_0) \right) \right| \\
  \leq |a|_\infty \int_{\mathbb{R}^N} \left| F( \tau_{-y_n}u_n) -
    F(\tau_{-y_n}u_n-u_0)-F(u_0) \right| =o(1)
\end{multline*}
and
\begin{multline*}
  \int_{\mathbb{R}^N} \left| \left( \tau_{-y_n} a \right) \left( F'(\tau_{-y_n} u_n) - F'(\tau_{-y_n} u_n-u_0) - F'(u_0) \right) \right|^{p/(p-1)} \\
  \leq |a|_\infty^{p/(p-1)} \int_{\mathbb{R}^N} \left| F'(\tau_{-y_n}
    u_n) - F'(\tau_{-y_n} u_n-u_0) - F'(u_0) \right|^{p/(p-1)}.
\end{multline*}
Recalling that the squared norm splits in the BL sense, the proof is
complete.
\end{proof}
\begin{definition} \label{def:4.3}
  A sequence $\{u_n\}_n \subset L^{s,2}(\mathbb{R}^N)$ vanishes if
  $\tau_{x_n} u_n \rightharpoonup 0$ in $L^{s,2}(\mathbb{R}^N)$ for any
  sequence $\{x_n\}_n$ of points in $\mathbb{R}^N$.
\end{definition}
\begin{remark} \label{rem:4.4}
  Any vanishing sequence is necessarily bounded in
  $L^{s,2}(\mathbb{R}^N)$, and by the Rellich-Kondratchev theorem (see
  \cite[Corollary 7.2]{DiNezza})
  $\tau_{x_n} u_n \to 0$ strongly in $L_{\mathrm{loc}}^2(\mathbb{R}^N)$ for every
  sequence $\{x_n\}_n \subset \mathbb{R}^N$. This yields that, for
  every $R>0$,
  \begin{equation*}
    \lim_{n \to +\infty} \sup \left\{
      \int_{B(x,R)} |u_n|^2 \mid x \in \mathbb{R}^N
      \right\} =0.
  \end{equation*}
  By the fractional version of Lions' vanishing lemma
  \cite[Proposition II.4]{Secchi12}, we deduce that $u_n \to 0$
  strongly in $L^q(\mathbb{R}^N)$ for every $2<q<2_s^\star$.
\end{remark}
\begin{definition}
  If $\{u_n\}_n$ is a sequence from $L^{s,2}(\mathbb{R}^N)$, we say that
  $\{DI_a(u_n)\}_n$ *-vanishes if
  $D I_{\tau_{x_n}a} (u_n) \rightharpoonup^\star 0$ in the weak*
  topology for every sequence $\{x_n\}_n \subset \mathbb{R}^N$.
\end{definition}
\begin{remark}
  It follows from the definition of the gradient and from the
  definition of the weak* topology that $\{DI_a(u_n)\}_n$ *-vanishes
  if, and only if, $\{\nabla I_a(u_n)\}_n$ vanishes in
  $L^{s,2}(\mathbb{R}^N)$ in the sense of Definition \ref{def:4.3}.
\end{remark}
\begin{lemma} \label{lem:4.7}
  Suppose that $\{u_n\}_n \subset L^{s,2}(\mathbb{R}^N)$, $\{y_n\}_n
  \subset \mathbb{R}^N$ and $a^* \in L^\infty(\mathbb{R}^N)$ are such
  that $\{DI_a(u_n)\}_n$ *-vanishes, $\tau_{-y_n} u_n \rightharpoonup
  u_0$ weakly in $L^{s,2}(\mathbb{R}^N)$ and $\tau_{-y_n}a
  \rightharpoonup^* a^*$ weakly*. If $v_n = u_n - \tau_{y_n} u_0$,
  then
  \begin{align}
    \lim_{n \to +\infty} \left( I_a(u_n) - I_a(v_n) \right) &= I_{a^*}
                                                              (u_0) \label{eq:4.3}
    \\
    \lim_{n \to +\infty} \left( \|u_n\|_{L^{s,2}}^2 - \|v_n\|_{L^{s,2}}^2
    \right) &= \|u_0\|_{L^{s,2}}^2 \label{eq:4.4} \\
    DI_{a^*}(u_0) &= 0. \label{eq:4.5}
  \end{align}
  Furthermore, also $\{DI_a (v_n)\}_n$ *-vanishes.
\end{lemma}
\begin{proof}
From the assumption that $\tau_{-y_n}a
  \rightharpoonup^* a^*$ we deduce that $I_{a^*}(u_0) =
  I_{\tau_{-y_n}}(u_0)+o(1)$.  Combining with Lemma \ref{lem:4.2} we
  get \eqref{eq:4.3}. Equation \eqref{eq:4.4} follows from the
  splitting properties of the squared norm. We prove now
  \eqref{eq:4.5}.

  Fix any $v \in L^{s,2}(\mathbb{R}^N)$. We have that
  $\lim_{n \to +\infty} F'(\tau_{-y_n}u_n)v = F'(u_0)v$ in
  $L^1(\mathbb{R}^N)$ due to the fact that $\tau_{-y_n} u_n \to u_0$
  strongly in $L_{\mathrm{loc}}^p(\mathbb{R}^N)$ (see again
  \cite{DiNezza}). Therefore
  \begin{multline*}
    DI_{a^*}(u_0)[v] = \langle u_0 \mid v \rangle -
    \int_{\mathbb{R}^N} \tau_{-y_n} a \, F'(u_0)v+ o(1) \\
    = \langle \tau_{-y_n}u_n \mid v \rangle - \int_{\mathbb{R}^N}
    \tau_{-y_n} a \, F'(\tau_{-y_n}u_n)v + o(1) \\
    = DI_{\tau_{-y_n}a}(\tau_{-y_n}u_n)[v]+o(1) = o(1),
  \end{multline*}
  where we have used the assumption that $\{DI_{a}(u_n)\}_n$
  *-vanishes. This completes the proof of \eqref{eq:4.5}.

  To conclude the proof, we suppose that $\{x_n\}_n$ is a sequence of
  points from $\mathbb{R}^N$ and that $v \in
  L^{s,2}(\mathbb{R}^N)$. We distinguish two cases.
  \begin{itemize}
  	\item[(i)] Up to a subsequence, $\lim_{n \to +\infty} |x_n+y_n| = +\infty$. This implies that $\tau_{-x_n-y_n} v \rightharpoonup 0$ weakly in $L^{s,2}(\mathbb{R}^N)$, and thus $F'(u_0) \tau_{-x_n-y_n} v \to 0$ strongly in $L^1(\mathbb{R}^N)$. This yields
  	\begin{equation} \label{eq:4.6}
  	D I_{\tau_{-y_n}a} (u_0) [\tau_{-x_n-y_n} v] = o(1).
  	\end{equation}
  	Equation \eqref{eq:4.6}, Lemma \ref{lem:4.2} and the fact that $\{DI_a(v_n)\}_n$ *-vanishes, we obtain
  	\begin{align*}
  	D I_{\tau_{x_n}a}(\tau_{x_n} v_n)[v] &= D I_{\tau_{-y_n}a} (\tau_{-y_n} v_n)[\tau_{-x_n-y_n}v] \\
  	&= D I_{\tau_{-y_n}a} (\tau_{-y_n} u_n) [\tau_{-x_n-y_n}v] - D I_{\tau_{-y_n}a}(u_0)[\tau_{-x_n-y_n}v]+o(1) \\
  	&= D I_{\tau_{-y_n}a} (\tau_{-y_n}u_n)[\tau_{-x_n-y_n} v]+o(1)\\
  	&= D I_{\tau_{x_n}a} (\tau_{x_n} u_n)[v]+o(1) \\
  	&=o(1).
  	\end{align*}
  	Since the limit is independent of the subsequence, this shows that $\{D I_a (v_n)\}_n$ *-vanishes in this case.
  	\item[(ii)] Up to a subsequence, $\lim_{n \to +\infty} \left( x_n+y_n \right) = -\xi \in \mathbb{R}^N$. In this case,
  	\begin{align*}
  	D I_{\tau_{x_n}a}(\tau_{x_n} v_n)[v] &= D I_{\tau_{-y_n}a} (\tau_{-y_n} v_n)[\tau_\xi v] + o(1) \\
  	&= D I_{\tau_{-y_n}a} (\tau_{-y_n} u_n)[\tau_{\xi}] - D I_{\tau_{-y_n}a}(u_0)[\tau_{\xi} v]+o(1) \\
  	&= -D I_{\tau_{-y_n}a}(u_0)[\tau_{\xi} v]+o(1) \\
  	&=-D I_{a^*}(u_0)[\tau_{\xi}v]+o(1) \\
  	&=o(1),
  	\end{align*}
  	and we conclude as before.
  \end{itemize}
\end{proof}
\begin{proposition} \label{prop:4.8}
Let $\{u_n\}_n$ be a Palais-Smale sequence for $I_a$ at level $c \in \mathbb{R}$. One of the following alternatives must hold:
\begin{itemize}
\item[(a)] $\lim_{n \to +\infty} u_n=0$ strongly in $L^{s,2}(\mathbb{R}^N)$;
\item[(b)] after passing to a subsequence, there exist a positive integer $k$, $k$ sequences $\{y_n^i\}_n \subset \mathbb{R}^N$, $k$ functions $a^i \in L^\infty(\mathbb{R}^N)$, and $k$ functions $u^i \in L^{s,2}(\mathbb{R}^N) \setminus \{0\}$ for $i=1,\ldots,k$ such that $DI_{a^i}(u^i)=0$ for every $i=1,\ldots,k$ and such that the following hold true:
\begin{equation} \label{eq:4.7}
\lim_{n \to +\infty} \left\Vert u_n-\sum_{i=1}^k \tau_{y_n^i} u^i \right\Vert_{L^p} =0,
\end{equation}
\begin{equation}\label{eq:4.8}
c \geq \sum_{i=1}^k I_{a^i}(u^i),
\end{equation}
\begin{equation}
\lim_{n \to +\infty} \tau_{-y_n^i}a = a^i \quad\hbox{in the weak* topology},
\end{equation}
and
\begin{equation}
\lim_{n \to +\infty} \left| y_n^i - y_n^j \right| = +\infty \quad\hbox{if $i \neq j$}.
\end{equation}
\end{itemize}
\end{proposition}
\begin{proof}
It follows from the assumptions that the sequence $\{u_n\}_n$ is bounded in $L^{s,2}(\mathbb{R}^N)$ and $\{DI_a(u_n)\}_n$ *-vanishes. We distinguish two cases. 

If $\{u_n\}_n$ vanishes, then by Remark \ref{rem:4.4} $\{u_n\}_n$ converges strongly to zero in $L^p(\mathbb{R}^N)$. Recalling that $DI_a(u_n)[u_n]=o(1)$, we conclude that $\{u_n\}_n$ converges to zero strongly in $L^{s,2}(\mathbb{R}^N)$.

If, on the contrary, $\{u_n\}_n$ does not vanish, then there exist a function $u^1 \in L^{s,2}(\mathbb{R}^N)$ and a sequence $\{y^1_n\}_n \subset \mathbb{R}^N$ such that, after passing to a subsequence, and writing $u^1_n = u_n$, we have $\tau_{-y^1_n}u^1_n \rightharpoonup u^1$ weakly. Recalling that $\mathscr{P}$ is compact, we may also assume that $\{\tau_{-y^1_n} a\}_n$ weakly* converges to $a^1 \in L^\infty(\mathbb{R}^N)$. We then define $u^2_n = u^1_n-\tau_{y^1_n} u^1$, so that $\tau_{-y^1_n} u_n^2 \rightharpoonup 0$ weakly. 

Lemma \ref{lem:4.7} ensures that 
\begin{gather*}
\lim_{n \to +\infty} I_a(u^1_n)-I_a(u^2_n) = I_{a^1} (u^1),\\
\lim_{n \to +\infty} \left\| u^1_n \right\|^2_{L^{s,2}} - \left\| u^2_n \right\|^2_{L^{s,2}} =0, \\
DI_{a^1} (u^1) =0
\end{gather*}
and $\{DI_a (u^2_n)\}_n$ *-vanishes. If $\{u^2_n\}_n$ vanishes, then it converges to zero in $L^p(\mathbb{R}^N)$ and thus also $\{u^1_n-\tau_{y^1_n} u^1\}_n$ converges to zero in $L^p(\mathbb{R}^N)$. Otherwise there exist $a^2 \in L^\infty(\mathbb{R}^N)$, $u^2 \in L^{s,2}(\mathbb{R}^N) \setminus \{0\}$ and a sequence $\{y^2_n\}_n \subset \mathbb{R}^N$ such that, up to a subsequence, $\lim_{n \to +\infty} \tau_{-y^2_n} a = a^2$ weakly* and $\lim_{n \to +\infty} \tau_{-y^2_n} u^2_n =u^2$ weakly. Necessarily, $\lim_{n \to +\infty} |y^1_n - y^2_n |=0$, since $\lim_{n \to +\infty} \tau_{-y^1_n} u^2_n=0$ weakly.

\medskip

Iterating this construction, we obtain sequences $\{y^1_n\}_n \subset \mathbb{R}^N$, functions $a^i \in L^\infty(\mathbb{R}^N)$ and functions $u^i \in L^{s,2}(\mathbb{R}^N) \setminus \{0\}$ for $i=1,2,3,\ldots$ Since each $u^i$ is a non-trivial critical point of $I_{a^i}$, we have that $(a^i)^{+} \neq 0$. On the other hand, $|(a^i)^{+}|_\infty \leq |a|_\infty$. Hence $u^i \in \mathscr{N}_{a^i}$ for every $i$ and by \eqref{eq:3.3} there exists a constant $C>0$, independent of $i$, such that $\|u^i\|_{L^{s,2}} \geq C$. For every $j$ we also have
\begin{equation*}
0 \leq \|u_{n}^{j+1}\|_{L^{s,2}}^2 = \|u_n\|^2_{L^{s,2}} - \sum_{i=1}^j \|u^i\|_{L^{s,2}}^2 + o(1),
\end{equation*}
which implies that the iteration must stop after finitely many steps. Therefore there exists a positive integer $k$ such that $\{u_n^{k+1}\}_n$ vanishes, $\{u_n^{k+1}\}_n$ converges to zero strongly in $L^p(\mathbb{R}^N)$ and \eqref{eq:4.7} holds true. Similarly,
\begin{equation*}
-\int_{\mathbb{R}^N} a \left| u_n^{k+1} \right|^p \leq I_a(u_n^{k+1}) = I_a(u_n) - \sum_{i=1}^k I_{a^i} (u^i) + o(1),
\end{equation*}
and also \eqref{eq:4.8} follows from $c = \lim_{n \to +\infty} I_a(u_n)$. The proof is complete.
\end{proof}

\section{Existence of a ground state}

The proof of the following comparison lemma is probably known, but we reproduce here for the reader's convenience.
\begin{lemma} \label{lem:5.1}
	Suppose that $a_1$, $a_2 \in L^\infty(\mathbb{R}^N)$. If $a_1 \geq a_2$, then $c_{a_1} \leq c_{a_2}$. If, in addition, $a_1 \neq a_2$ and $I_{a_2}$ possesses a ground state, then $c_{a_1}<c_{a_2}$. 
\end{lemma}
\begin{proof}
	Without loss of generality, we assume that $a_2^{+} = \max\{a_2,0\}$ is not identically equal to zero, otherwise there is nothing to prove. If $u \in \mathscr{N}_{a_2}$, then
	\begin{equation*}
	\int_{\mathbb{R}^N} a_1 |u|^p \geq \int_{\mathbb{R}^N} a_2 |u|^p >0.
	\end{equation*}
	We can therefore define
	\begin{equation} \label{eq:5.1}
	t = \left(
		\frac{\int_{\mathbb{R}^N} a_2 |u|^p}{\int_{\mathbb{R}^N} a_1 |u|^p}
	\right)^{1/(p-2)} \leq 1.
	\end{equation}
	Then we have
	\begin{equation*}
	DI_{a_1}(tu)[tu] = t^2 \left( \|u\|_{L^{s,2}}^2 - t^{p-2} \int_{\mathbb{R}^N} a_1 |u|^p \right) = t^2 D I_{a_2}(u)[u]=0,
	\end{equation*}
	and hence $tu \in \mathscr{N}_{a_1}$. Since
	\begin{equation*}
	I_{a_2}(u) = \frac{1}{2} \|u\|_{L^{s,2}}^2 - \frac{1}{p} \int_{\mathbb{R}^N} a_2 |u|^p = \left( \frac{1}{2} - \frac{1}{p} \right) \|u\|_{L^{s,2}}^2 \geq \left( \frac{1}{2} - \frac{1}{p} \right) \|tu\|_{L^{s,2}}^2 = J_{a_1}(u) \geq c_{a_1},
	\end{equation*}
	we conclude that $c_{a_2} = \inf_{u \in \mathscr{N}_{a_2}}I_{a_2}(u) \geq c_{a_1}$. Furthermore, if $a_1 \neq a_2$ (as elements of $L^\infty(\mathbb{R}^N)$) and $u$ is a ground state of $I_{a_2}$, then $|u|>0$. In \eqref{eq:5.1} we then have $t<1$, and it follows that $c_{a_2}=I_{a_2}(u) > I_{a_1}(tu) \geq c_{a_1}$.
\end{proof}
%
%
Recall the definition \eqref{eq:a-bar} of \(\bar{a}\). We have
\begin{proposition} \label{prop:5.3}
There results
\begin{align*}
c_{a} < c_{\bar{a}}.
\end{align*}
\end{proposition}
\begin{proof}
We first consider (i) of assumption \textbf{(A)}. Since \(\bar{a}\leq 0\), we have \(c_{\bar{a}}=\infty\). But \(c_{a} \in \mathbb{R}\) because \(a^{+} \neq 0\), and there is nothing more to prove. We can assume that \(\bar{a}>0\) in the rest of the proof. If (ii) of assumption \textbf{(A)} holds, recalling that \(\bar{a}>-\infty\) entails \(\mathscr{B}\neq \emptyset\) we can conclude that \(a \neq \bar{a}\). Now Lemma \ref{lem:5.1} implies that \(c_a < c_{\bar{a}}\), since \(I_{\bar{a}}\) has a ground state by the arguments of \cite[Theorem 1.1]{Ambrosio}.
\end{proof}
We are now ready to prove our main existence result.
\begin{proof}[Proof of Theorem \ref{th:main}]
We have \(\mathscr{N}_{a} \neq \emptyset\) and \(c_{a}<\infty\) because \(a^{+} \neq 0\). From \eqref{eq:3.3} we get \(c_{a}>0\). An application of Ekeland's Principle yields in a standard way a mimnimizing sequence \(\{u_n\}_n \subset \mathscr{N}_{a}\) for the functional \(\bar{I}_{a}\) defined as the restriction of \(I_{a}\) to \(\mathscr{N}_{a}\). This sequence is also a (PS)-sequence for \(\bar{I}_{a}\) at the level \(c_{a}\). By Proposition \ref{prop:3.7} \(\{u_n\}_n\) is a (PS)-sequence for \(I_{a}\) at the level \(c_{a}\). The strong convergence of \(\{u_n\}_n\) to zero is easily ruled out, since \(I_{a}(u_n) \to c_{a}>0\). Proposition \ref{prop:4.8} yields then a number \(k \in \mathbb{N}\), functions \(a^i \in \overline{\mathscr{A}}\) and non-trivial critical points \(u^i\) of \(I_{a^i}\) such that 
\begin{align*}
	c_{a} \geq \sum_{i=1}^k I_{a^i}(u^i).
\end{align*}
From the knowledge that each \(u^i\) is a non-trivial critical point of \(I_{a^i}\) we deduce \((a^i)^{+}\neq 0\) for every \(i=1,\ldots,k\). Again by \eqref{eq:3.3} we get \(I_{a^i}(u^i)>0\) for every \(i=1,\ldots,k\).

Suppose that for \emph{some} index \(i\) there results \(a^i \in \mathscr{B}\). Then \(a^i \leq \bar{a}\), and Lemma \ref{lem:5.1} together with Proposition \ref{prop:5.3} yield 
\(I_{a^i}(u^i) \geq c_{a^i} \geq c_{\bar{a}}>c_{a}\). This is a contrdiction. Therefore each \(a^i\) is a translation of \(a\), and \(I_{a^i}(u^i) \geq c_{a}\) for every \(i=1,\ldots,k\). This forces \(k=1\), and a translation of \(u^1\) is a ground state of \(I_a\).
\end{proof}

\section{An example}

Assumption \textbf{(A)} can be rephrased in a more familiar way for continuous bounded potentials.
\begin{proposition}
For any $a \in L^\infty(\mathbb{R}^N)$, define
\begin{align*}
	\hat{a} = \lim_{R \to +\infty} \operatorname*{ess\, sup}_{x \in \mathbb{R}^N \setminus B(0,R)} a(x).
\end{align*}
If \textbf{(A)} holds true with $\bar{a}$ replaced by $\hat{a}$, then \textbf{(A)} holds true with $\bar{a}$.
\end{proposition}
\begin{proof}
If \(\mathscr{B} = \emptyset\), then $\bar{a}=-\infty$ and \textbf{(A)} holds true. We may assume that \(\mathscr{B} \neq \emptyset\), so that \(a\) cannot be constant. Let us prove that
\begin{align}
\bar{a} \leq \hat{a}. \label{eq:6.1}
\end{align}
Pick $b \in \mathscr{B}$. There is a sequence $\{x_n\}_n \subset \mathbb{R}^N$ such that $\tau_{x_n}a \rightharpoonup^\star b$. Translations are continuous in the weak$^\star$ topology of $L^\infty(\mathbb{R}^N)$, since they are continuous in $L^1(\mathbb{R}^N)$. For the sake of contradiction, suppose that $\{x_n\}_n$ contains a bounded subsequence. Up to a further subsequence, there must exist a point $\xi\in\mathbb{R}^N$ such that $x_n \to \xi$ and $\tau_{x_n} a \rightharpoonup^\star \tau_{\xi}a$. Since $\mathscr{P}$ is metrizable, $\tau_{\xi} a = b \notin \mathscr{A}$, a contradiction. Therefore $\lim_{n \to +\infty} |x_n|=+\infty$. 

Let $\varepsilon>0$ be given, and apply Lemma \ref{lem:3.1}: there exists $\varphi \in L^1(\mathbb{R}^N)$ with $\varphi \geq 0$ and $\|\varphi\|_{L^1}=1$ such that
\begin{align*}
	\int_{\mathbb{R}^N} b \varphi \geq \operatorname{ess\, sup}b - \frac{\varepsilon}{2}.
\end{align*}
Choose $\tilde{\psi} \in C_c^\infty(\mathbb{R}^N)$ such that $\tilde{\psi} \geq 0$ and
\begin{align*}
	\left\| \varphi - \tilde{\psi} \right\|_{L^1} \leq \frac{\varepsilon}{4 \|b\|_{L^\infty}}.
\end{align*}
Now $\psi = \tilde{\psi}/\|\tilde{\psi}\|_{L^1} \in C_c^\infty(\mathbb{R}^N)$ satisfies
\begin{align*}
	\left\| \varphi - \psi \right\|_{L^1} \leq \frac{\varepsilon}{2 \|b\|_{L^\infty}},
\end{align*}
$\psi \geq 0$ and $\|\psi\|_{L^1} = 1$. This implies
\begin{align*}
	\int_{\mathbb{R}^N} b \psi = \int_{\mathbb{R}^N} b \varphi - \int_{\mathbb{R}^N} b (\varphi-\psi) \geq \int_{\mathbb{R}^N} b \varphi - \|b\|_{L^\infty} \left\|\psi - \varphi \right\|_{L^1} \geq \operatorname{ess\, sup}b - \varepsilon.
\end{align*}
Suppose that $\operatorname{supp}\psi \subset B(0,R)$: then
\begin{multline*}
\operatorname{ess\, sup}b - \varepsilon \leq \int_{\mathbb{R}^N} b \psi = \lim_{n \to +\infty} \int_{\mathbb{R}^N} (\tau_{x_n}a)\psi \\
\leq \lim_{n \to +\infty} \operatorname*{ess\, sup}_{x \in B(-x_n,R)} a(x) \int_{\mathbb{R}^N} \psi \leq \lim_{n \to +\infty} \operatorname*{ess\, sup}_{x \in \mathbb{R}^N \setminus B(0,|x_n|-R)} a(x) = \hat{a}.
\end{multline*}
Since $\varepsilon>0$ is arbitrary, we conclude that $\operatorname{ess\, sup}b \leq \hat{a}$.
If (i) of assumption \textbf{(A)} holds, then \eqref{eq:6.1} yields $\bar{a} \leq \hat{a} \leq 0$. If (ii) holds, then \eqref{eq:6.1} yields $\bar{a} \leq \hat{a} \leq a$, and the proof is complete.
\end{proof}
An immediate consequence of Theorem \ref{th:main} is then the following.
\begin{corollary}
If \(a\) is a bounded continuous function such that either \(\limsup_{|x| \to +\infty} a(x) \leq 0\) or \(\limsup_{|x| \to +\infty} a(x) \leq a\), then equation \eqref{eq:1.1} has (at least) a positive ground state as soon as $2<p<2_s^\star$.
\end{corollary}

\section*{Acknowledgement}

The author is member of the {\em Gruppo Nazionale per
l'Analisi Ma\-te\-ma\-ti\-ca, la Probabilit\`a e le loro Applicazioni}
(GNAMPA) of the {\em Istituto Nazionale di Alta Matematica} (INdAM).
The manuscript was realized within the auspices of the INdAM -- GNAMPA
Projects {\em Problemi non lineari alle derivate parziali}
(Prot\_U-UFMBAZ-2018-000384).


\begin{thebibliography}{99}

\bibitem{AckerChagoya} \textsc{N. Ackermann, J. Chagoya}, \emph{Ground
    states for irregular and indefinite superlinear Schr\"{o}dinger
    equations}, J. Differential Equations \textbf{261} (2016),
  5180--5201.
  
  \bibitem{Adams} \textsc{D. R. Adams, L. I. Hedberg}. Function spaces and potential
theory. Grundlehren der Mathematischen Wissenschaften \textbf{314}. Springer--Verlag, Berlin, 1996.
  
\bibitem{Ambrosio} \textsc{V. Ambrosio}, \emph{Ground states solutions for a non-linear equation involving a pseudo-relativistic Schr\"{o}dinger operator}, Journal of Mathematical Physics \textbf{57}, 051502 (2016).

\bibitem{BahriLions} \textsc{A. Bahri, P.-L. Lions}, \emph{On the existence of a positive solution of semilinear elliptic equations in unbounded domains}, Ann. Inst. H. Poincar\'{e} Anal. Non Lin\'{e}aire \textbf{14} (1997), no. 3, 365--413.

\bibitem{ClappWeth} \textsc{M. Clapp, T. Weth}, \emph{Multiple solutions of nonlinear scalar field equations}, Comm. Partial Differential Equations \textbf{29} (2004), 1533--1554.

  \bibitem{DiNezza} \textsc{E. Di Nezza, G. Palatucci, E. Valdinoci},
    \emph{Hitchhiker's guide to fractional Sobolev spaces},
    Bull. Sci. math. \textbf{136} (2012), 521--573.
    
   \bibitem{BerestyckiLionsI}  \textsc{H. Berestycki, P.-L. Lions}, \emph{Nonlinear scalar field equations. I. Existence of a ground state}. Arch. Rational Mech. Anal. 82 (1983), no. 4, 313–345.
   
   \bibitem{BerestyckiLionsII}  \textsc{H. Berestycki, P.-L. Lions}, \emph{Nonlinear scalar field equations. II. Existence of infinitely many solutions}. Arch. Rational Mech. Anal. 82 (1983), no. 4, 313–345.
   
  \bibitem{CeramiMolle} \textsc{G. Cerami, R. Molle}, \emph{On some Schr\"{o}dinger equations with non regular potential at infinity}, Discrete Contin. Dyn.
Syst. \textbf{28} (2) (2010) 827--844,
   
  \bibitem{CingolaniSecchi15} \textsc{S. Cingolani, S. Secchi}, \emph{Ground states for the pseudo-relativistic Hartree equation with external potential}. Proceedings of the Royal Society of Edinburgh Section A: Mathematics, \textbf{145} (1), 73--90.
   
\bibitem{CingolaniSecchi18}  \textsc{S. Cingolani, S. Secchi}, \emph{Intertwining solutions for magnetic relativistic Hartree type equations}. Nonlinearity, \textbf{31} (5), 2294.
   
   \bibitem{DingNi} \textsc{W. Y. Ding, W.-M. Ni}, \emph{On the existence of positive entire solutions of a semilinear elliptic equation}, {Arch. Rational Mech. Anal.} \textbf{91} (1986), 283--308.

\bibitem{Fall} \textsc{M. M. Fall, V. Felli}, \emph{Unique continuation properties for relativistic Schr\"{o}dinger operators with a singular potential},  Discrete Contin. Dyn. Syst. \textbf{35} (2015), no. 12, 5827–5867

\bibitem{Felmer} \textsc{P. Felmer, A. Quaas, J. Tan}, \emph{Positive solutions of the
nonlinear Schr\"odinger equation with the fractional Laplacian}, Proc. Roy. Soc. Edinburgh Sect. A \textbf{142} (2012), 1237--1262.

\bibitem{FelmerVergara} \textsc{P. Felmer, I. Vergara}, \emph{Scalar field equation with non-local diffusion}, Nonlinear Differ. Eq. Appl. \textbf{22} (2015), 1411--1428.

\bibitem{Rabinowitz} \textsc{P. H. Rabinowitz}, \emph{On a class of nonlinear Schr\"{o}dinger equations}, Z. Angew. Math. Phys. \textbf{43} (1992), 270--291.
 
\bibitem{Rudin} \textsc{W. Rudin}, Functional analysis, second
  edition, McGraw-Hill Inc., New York, 1991.

  \bibitem{Secchi12} \textsc{S. Secchi}, \emph{Ground state solutions
      for nonlinear fractional Schr\"{o}dinger equations in
      \(\mathbb{R}^N\)}, Journal of Mathematical Physics \textbf{54}
    031501 (2013).

\bibitem{Secchi17} \textsc{S. Secchi}, \emph{Concave-convex
    nonlinearities for some nonlinear fractional equations involving
    the Bessel operator},  Complex Variables and Elliptic Equations, \textbf{62} (2017), 654-669.
    
    \bibitem{Secchi17-1}  \textsc{S. Secchi}, \emph{On some nonlinear fractional equations involving the Bessel operator}, J. Dynam. Differential Equations \textbf{29} (2017), no. 3, 1173–1193.
    
    \bibitem{Secchi17-2} \textsc{S. Secchi}, \emph{On a generalized pseudorelativistic Schrödinger equation with supercritical growth}, arXiv:1708.03479.

\bibitem{Stein} \textsc{E. M. Stein}. Singular integrals and differentiability properties of functions. Princeton Mathematical Series, No. 30. Princeton University Press, Princeton, N.J. 1970.

\end{thebibliography}
\end{document}